\documentclass[12pt]{amsart}
\usepackage{amsthm,amsmath,amssymb,amscd,graphics,enumerate, stmaryrd,xspace,verbatim, epic, eepic,url,fullpage}

\usepackage[usenames,dvipsnames]{color}

\usepackage[colorlinks=true]{hyperref}

\usepackage[active]{srcltx}
\usepackage[all]{xypic}
\SelectTips{cm}{}

{
   \newtheorem{theorem}[subsubsection]{Theorem}
      \newtheorem*{theorem*}{Theorem}
   \newtheorem{proposition}[subsubsection]{Proposition}

   \newtheorem{lemma}[subsubsection]{Lemma}
   \newtheorem*{claim}{Claim}

   \newtheorem{corollary}[subsubsection]{Corollary}
   
   \newtheorem*{conjecture*}{Conjecture}

}
{\theoremstyle{definition}
          \newtheorem*{exercise*}{Exercise}

   \newtheorem*{example*}{Example}
   \newtheorem{definition}[subsubsection]{Definition}
   
   \newtheorem{problem}[subsubsection]{Problem}
   \newtheorem*{definition*}{Definition}
   \newtheorem{rem}[subsubsection]{Remark}

}
\newcommand{\PP}{{\mathbb{P}}}
\newcommand{\ZZ}{{\mathbb{Z}}}

\newcommand{\GG}{{\mathbb{G}}}

\newcommand{\bbA}{{\mathbb{A}}}
\renewcommand{\AA}{{\mathbb{A}}}

\def\uy{{\underline{y}}}
\def\ut{{\underline{t}}}

\def\Bl{{\rm Bl}}
\def\rmlog{{\rm log}}

\def\hatcO{{\widehat\cO}}

\newcommand{\bmu}{{\boldsymbol{\mu}}}

\newcommand{\tC}{{\widetilde C}}
\newcommand{\tX}{{\widetilde X}}
\newcommand{\tY}{{\widetilde Y}}

\newcommand{\cC}{{\mathcal C}}
\renewcommand{\cD}{{\mathcal D}}

\newcommand{\cG}{{\mathcal G}}

\newcommand{\cI}{{\mathcal I}}
\newcommand{\cJ}{{\mathcal J}}

\newcommand{\cM}{{\mathcal M}}

\newcommand{\cO}{{\mathcal O}}

\newcommand{\cX}{{\mathcal X}}
\newcommand{\cY}{{\mathcal Y}}

\newcommand{\nc}{{\operatorname{nc}}}

\def\<{\langle}
\def\>{\rangle}

\newcommand{\Spec}{\operatorname{Spec}}

\newcommand{\Hom}{{\operatorname{Hom}}}

\newcommand{\Aut}{{\operatorname{Aut}}}

\newcommand{\chara}{\operatorname{char}}

\newcommand{\ocM}{\overline{{\mathcal M}}}

\newcommand{\oX}{{\overline{X}}}

\def\:{{\colon}}
\def\.{{,\dots,}}
\def\dim{{\rm dim}}

\def\inv{{\rm inv}}
\def\loginv{{\rm loginv}}

\newcommand{\double}{\genfrac..{0pt}1
{\raise -1pt\hbox{$\scriptstyle\longrightarrow$}}{\raise 3pt\hbox
{$\scriptstyle\longrightarrow$}}}

\renewcommand{\setminus}{\smallsetminus}

\def\nor{{\rm nor}}
\def\int{{\rm int}}

\def\tototi{\mathbin{\mathop{\otimes}\limits^{\raise-1pt\hbox
{$\scriptscriptstyle {\rm L}$}}}}

\def\indlim{\mathop{\vrule width0pt height7pt depth
4pt\smash{\lim\limits_{\raise 1pt\hbox to 14.5pt
{\rightarrowfill}}}}}
\def\projlim{\mathop{\vrule width0pt height7pt depth
4pt\smash{\lim\limits_{\raise 1pt\hbox to 14.5pt
{\leftarrowfill}}}}}

\newcommand\displaceamount{3pt}

\newcommand{\doubledown}{\ar@<\displaceamount>[d]\ar@<-\displaceamount>[d]}

\newcommand{\doubleup}{\ar@<\displaceamount>[u]\ar@<-\displaceamount>[u]}

\newcommand{\doubleright}{\ar@<\displaceamount>[r]\ar@<-\displaceamount>[r]}

\newcommand{\thickslash}{\mathbin{\!\!\pmb{\fatslash}}}

\def\into{\hookrightarrow}

\newcommand{\ord}{{\operatorname{ord}}}

\def\tilp{{\widetilde p}}

\def\tilC{{\widetilde C}}

\def\tilX{{\widetilde X}}

\def\toisom{\xrightarrow{{_\sim}}}

\def\hatK{{\widehat K}}

\begin{document}
\title{Partial desingularization up to  normal-crossings\\ in characteristic 0 and 2}

\author[D. Abramovich]{Dan Abramovich}
\address{Department of Mathematics, Box 1917, Brown University,
Providence, RI, 02912, U.S.A}
\email{dan\_abramovich@brown.edu}

\author[M. Temkin]{Michael Temkin}
\address{Einstein Institute of Mathematics\\
               The Hebrew University of Jerusalem\\
                Edmond J. Safra Campus, Giv'at Ram, Jerusalem, 91904, Israel}
\email{temkin@math.huji.ac.il}

%

\thanks{This research is supported by BSF grant and 2022230 and NSF grant DMS-2401358, M.T. was visiting IAS, IHES and MPIM when this work was done and he is grateful to these institutes for the hospitality}

\date{\today}

\maketitle
\setcounter{tocdepth}{1}
\tableofcontents

\section{Introduction}

\subsection{Normal-crossing preserving resolutions}

The resolution of algebraic \emph{surface} singularities, especially in characteristic 0, is well understood, and yet it continues to have lessons to teach us.

Let $X$ be an algebraic variety.
We recall that a point $x \in X$ is a \emph{normal crossings point} if, after passing to an \'etale (or analytic, or formal) neighborhood $x \in U$, there is an embedding $U \subset Y$ and a system of regular parameters $y_1,\ldots,y_n$ at $x\in Y$ so that $U$ is defined by the equation
\begin{equation} \label{Eq:nc} y_1 \cdots y_k = 0\end{equation}
 for some $k\leq n$. To emphasize $k$ one says that equation  \eqref{Eq:nc} describes an $\nc(k)$ point \cite{BDS-B-RL}. To avoid spelling out all the necessary choices, we  say that $X$ has \emph{local equation} \eqref{Eq:nc}.

Define a \emph{normal-crossing-preserving resolution} 
$X'\to X$ to be a proper birational morphism, which is an isomorphism on the locus of normal crossings points of $X$, and where $X'$ has only normal crossings singularities.  As a shorthand we use the term \emph{NC-preserving resolution}.

The main thrust of this work is to show that while things work well in characteristic 0, the situation in characteristic 2 is very different and poses new challenges.

\subsection{The trouble of pinch points}
Already in dimension 2 we have:

\begin{lemma}[Koll\'ar {\cite[\S 8]{Kollar-semilog}}]\label{Lem:unresolve-intro}
Suppose $X$ is a surface and $p\in X$ a pinch point. Then there exists no NC-preserving resolution  $X'\to X$ by a variety $X'$.
\end{lemma}

We restate this  and recall the argument, see Section \ref{Sec:unresolve}, as its principle reappears below (Theorem \ref{Th:wild-pp}\eqref{It:wild-pp-no-tame}).

For the present discussion, a \emph{pinch point} in characteristic 0 has local equation, in the sense above,
\begin{equation} \label{Eq:pp} x^2 - y^2z=0\end{equation}

The surface \eqref{Eq:pp} is known as \emph{Whitney's umbrella}\footnote{As Eleonore Faber's picture \url{https://www.researchgate.net/figure/The-Whitney-Umbrella_fig6_243073708} suggests, it is an umbrella in distress. A competing view, suitable for couples with a love-hate relationship, is given in \url{https://virtualmathmuseum.org/Surface/whitney_umbrella/whitney_umbrella.html}}. Beyond its inherent beauty, it has a fascinating story to tell, see Section \ref{Sec:umbrellas}.

\subsection{Weighted NC-preserving resolutions exists in characteristic 0} In contrast, if one allows for \emph{stack theoretic weighted blowups}, we have

\begin{theorem}[{Belotto da Silva--Bierstone  \cite{BDS-B-nc}, W{\l}odarczyk  \cite{Wlodarczyk-nc}, see also \cite{AT-partial, BBB}}]\label{Th:nc} Let $X$ be a pure dimensional variety over a field of characteristic 0. There is a functorial stack theoretic NC-preserving resolution  $X'\to X$.
\end{theorem}

In Section \ref{Sec:nc0} we  describe a principle that makes such results inevitable, which we discovered at the same time as \cite{BDS-B-nc}, \cite{Wlodarczyk-nc}, in particular providing a short proof of Theorem    \ref{Th:nc}.

Numerous other results of similar nature are proven in \cite{Szabo,B-M-except-I,B-L-M-minimal2,BDSMV,BDS-B-RL}.

As stated, Theorem \ref{Th:nc}  answers a question by Koll\'ar \cite[Problem 9]{Kollar-semilog} if one allows for resolution by Deligne--Mumford stacks with abelian stabilizers. Passing to coarse moduli spaces directly, this answers the same question with the class including abelian quotients of normal crossings singularities. Applying relative destackification we may narrow the singularities on coarse moduli spaces further.

\begin{definition}\label{Def:poly-circulant}
We say that the coarse moduli  scheme $\overline X$ of $X$ has \emph{higher pinch point singularities} if the stabilizer of a point $p\in X$ with local description  $x_1\cdots x_k = 0$ is abelian and embeds into the permutation group $\Sigma_k$ of the branches.
\end{definition}
Since $k$ is bounded by $\dim(X)+1$ and $\Sigma_k$ has finitely many representations of rank $\dim(X)+1$, there is a finite set of possible singularity types appearing in $\overline X$.
\begin{corollary}\label{Cor:poly-circulant} Let $X$ be a normal-crossings separated Deligne--Mumford stack with abelian inertia in characteristic 0. There is a sequence of blowups and root constructions, which are trivial over the locus where $X$ is representable and normal crossings,
resulting in a modification $X' \to X$ such that $\overline X'$ has higher pinch point  singularities.
\end{corollary}
Corollary \ref{Cor:poly-circulant} is proven in Section \ref{Sec:poly-circulant}.

 In Belotto da Silva and Bierstone's  result \cite{BDS-B-nc},  the singularities of the coarse moduli space of $X'$ are restricted to be certain \emph{hypersurface} higher pinch points called \emph{product circulant singularities}; they also control the exceptional locus. In \cite{Wlodarczyk-nc}, W{\l}odarczyk treats a more general class of \emph{non-reduced normal crossings singularities,} again controlling exceptional loci.

%
\begin{problem} Compare the product-circulant singularities of \cite{BDS-B-nc} to Definition \ref{Def:poly-circulant}.
\end{problem}

A project proving semistable reduction of varieties with normal-crossings singularities using Theorem \ref{Th:nc}, and aiming for a streamlined proof of properness of KSBA moduli spaces in characteristic 0, is underway.

\subsection{Characteristic 2}
 We now turn attention to a base field $k$ of characteristic 2. In Section \ref{Sec:inseparable-umbrellas} we briefly discuss the \emph{inseparable umbrella} given by equation \eqref{Eq:pp}. While this is not a true pinch point, it serves to teach us some important lessons, as shown in W{\l}odarczyk's \cite{Wlodarczyk-cobordant}.
A true pinch point, in any characteristic, is the seminormal singularity obtained by gluing a smooth curve $\tilde C$ on the normalization $X^\nor$ to itself by an involution having one fixed point lying over a single point  of a curve $C \subset X$. Statement \eqref{It:wild-pp-stack} below, which we see as a reassuring \emph{positive} statement, follows from this description, see Section \ref{Sec:wild-pp-stack}:

\begin{theorem}\label{Th:wild-pp}
\begin{enumerate}
\item\label{It:wild-pp-stack} Let $X$ be a surface in any characteristic with at most normal crossings and pinch points. Then  $X$ is the coarse moduli space of a normal-crossings stack $\cX$, where $\cX \to X$ is an isomorphism away from the pinch points, with $\ZZ/2\ZZ$ as stabilizers at all pinch points.
\item\label{It:wild-pp-no-tame}
If $\chara(k) =2$ and $X$ has at least one pinch point, then no sequence of weighted blowups gives a NC-preserving stack-theoretic resolution of $X$.
\end{enumerate}
\end{theorem}

Statement \eqref{It:wild-pp-no-tame} above is a new \emph{negative} statement showing that a type of resolution of singularities that holds in characteristic 0 simply does not hold in characteristic 2. In essence it follows since weighted blowups are all \emph{tame} stacks in the sense of \cite{AOV1}, with an argument  analogous to that of Lemma \ref{Lem:unresolve-intro}. See Section \ref{Sec:wild-pp-no-tame}.

\subsection{Classifying wild umbrellas}

The discussion above relies on abstract definitions and requires no explicit equations. A closer understanding the situation can be gained after passing to local fields and is discussed in Section \ref{Sec:wild-umbrellas} below. Assuming $k$ is algebraically closed, there is a sequence of distinct relevant pinch points, given by local equations as follows:

\begin{definition} Fix an integer $n\geq 1$. The \emph{$n$-th wild umbrella} is the surface $X_n\subset \AA^3$ given by
\begin{equation} \label{Eq:pp(n)} x^2 + y^2z + xyz^n=0\end{equation}
\end{definition}

Assuming one believes in \emph{embedded} resolution in positive characteristics, Theorem \ref{Th:wild-pp} suggests that one should seek to answer the following:
\begin{problem}\label{Prob:embedded}  Let $X \subset Y= \AA^3$ be an embedded surface with at most pinch points in characteristic 2. Find a stack-theoretic modification $\cY' \to Y$, with $\cY'$ smooth, which is an isomorphism away from the pinch points of $X$, inducing a NC-preserving resolution $\cX'\to X$.
\end{problem}

See further discussion focussing on $X_n$ of Equation \eqref{Eq:pp(n)} in section \ref{Sec:ambient-challenge}.

\subsection*{Acknowledgements} {We thank Andr\'e Belotto da Silva, Edward Bierstone, and Alberto Landi for their interest in the paper, and for pointing out serious errors in earlier versions.}

\section{The true tragic story of Whitney's umbrella}\label{Sec:umbrellas}

\subsection{The classical Whitney umbrella}

\subsubsection{The umbrella and the singular locus}\label{Zsec}
Let $k$ be a field such that $\chara(k)\neq 2$ and let $X=\Spec(k[x,y,z]/(x^2-y^2z)$ be the classical Whitney umbrella. The $z$-axis $C=V(x,y)$ is its singular locus and the origin $p=V(x,y,z)$ is the most singular point of $X$, the pinch point. Note that $X$ is a normal crossings singularity at any point of $C\setminus \{p\}$.

\subsubsection{The singularity is not improved under a standard blopwup}\label{Sec:unimprovable}
The Whitney umbrella provides a classical (in fact, the simplest) example of a singularity, which is not improved by a naive attack: if one blows up the pinch point, it shows up again on the new variety. Indeed, $X=V(x^2-y^2z)$ is a divisor in $Y=\Spec(k[x,y,z])$ and blowing up the origin and looking at the $z$-chart $Y'_z=\Spec(k[x',y',z])$, where $x'=\frac{x}{z}$, $y'=\frac{y}{z}$ we obtain that the full pullback is $V(z^2(x'^2-y'^2z))$ and the strict transform is $X'=V(x'^2-y'^2z)$, which is precisely the same Whitney umbrella. In particular, the singularity re-appears after the blowing up.

\subsubsection{What one does instead}
 In the classical principalization, which uses blowings up at smooth centers, an informal conclusion is that one cannot just improve the worst singularities each time, and any method must use some memory of the state of the algorithm. In particular, in the case of the Whitney umbrella one first blows up the pinch point, and only then blows up the whole $z$-axis on the transform, getting rid of the pinch point along with the original normal-crossings points. After the first blowup, the new $z$-axis $x'=y'=0$ has its pinch point on the exceptional divisor. An algorithm taking account of  history views such an exceptionally marked  pinch point ``as singular as normal crossings'' and blows them up together.

\subsubsection{It gets even worse} The argument in Section \ref{Sec:unimprovable} is quite convincing, but does not rule out existence of a ``very smart'' algorithm, which detects the difficulty and blows up the $z$-axis at once. This is not to be, as further examples   noticed by {Koll\'ar \cite[Example 3.6.2]{Kollar} and W{\l}odarczyk \cite{ATW-weighted, Wlodarczyk-cobordant} show}. Essentially the same computation shows that the singularity of $W=V(x^2-yzt)$ in $\bbA^4_k$ also persists after a blowing up.  Its singular locus $V(x,yz,yt,zt)$  is $S_3$-invariant, and the origin $p$ is the only smooth $S_3$-invariant subvariety in the singular locus that contains $p$. Therefore there exists no {\em functorial} memoryless algorithm which blows up only smooth centers.

\subsubsection{The normalization}
We denote by $\tilX$ the normalization of $X$, set $\tilC=C\times_X\tilX$ and let $\tilp$ be the preimage of $p$ in $\tilC$.

\begin{lemma}\label{norlem}
  Keep the above notation, then
  \begin{itemize}
	\item $\tilX=\Spec(k[y,w])$, where $w=\frac{x}{y}$ with $w^2=z$.
	\item $\tilC=\Spec(k[w])$ is a double cover of $C=\Spec(k[z])$, which is ramified at $\tilp$.
  \end{itemize}
\end{lemma}
\begin{proof}
The normal domain $k[y,w]$ is integral over $k[x,y,z]/(x^2-y^2z)$ and has the same field of fractions. This implies the first claim, and the second claim follows easily.
\end{proof}

\subsubsection{Unresolvability}\label{Sec:unresolve}
Set $X_0=X\setminus\{p\}$, $\tilX_0=\tilX\setminus\{\tilp\}$, $C_0=C\setminus\{p\}$ and $\tilC_0=\tilC\setminus\{\tilp\}$. The surface $X_0$ is normal crossings, so its normalization $\tilX_0\to X_0$ induces an \'etale covering $\tilC_0\setminus C_0$ of the singular locus. However, $X_0$ is not a \emph{simple} normal crossings surface, so the cover $\tilC_0\to C_0$ does not have to be split, and this indeed happens in our case. Moreover, it degenerates to a branched cover over the partial compactification $C$. This has the following consequence.

\begin{lemma}[Koll\'ar {\cite[\S 8]{Kollar-semilog}}]\label{unresolve}
Keep the above notation. Let $X'$ be a variety and $X'\to X$ be a modification which is an isomorphism outside of $p$. Then $p$ possesses a unique preimage $p'$ on the strict transform $C'$ of $C$ and $X'$ is not normal crossings at $p'$. In particular, there exists no NC-preserving resolution  $X'\to X$ by a variety.
\end{lemma}
\begin{proof}
Note that if $X$ is a normal crossings surface and $C$ is a \emph{smooth} irreducible component of its singular locus, then the normalization of $X$ restricts to an \'etale cover of $C$. Indeed, this claim is \'etale local, hence it suffices to check it in the simple normal-crossings case, which is trivial.

In the situation of the lemma, the map $C'\to C$ is an isomorphism over $C_0$ and $p$ is normal, hence $C'=C$ and $p$ has a unique preimage $p'$ in $C'$. By Lemma~\ref{norlem} the normalization map $\tilX'\to X'$ restricts to a non-split double cover $\tilC_0\to C_0$ over $C_0$ whose extension to the whole $C$ is branched over $p'$. Therefore, $X'$ is not normal crossings at $p'$.
\end{proof}

The lemma explains why the pinch point reappears on the $z$-chart of the blowing up $X'=\Bl_p(X)$, which is the chart containing $p'\in C'$. In fact, this is the best that can happen to such a point under a modification if $C_0$ is kept unchanged.

\subsubsection{Weighted resolvability}
The situation improves once one extends the category of varieties to orbifolds (or DM stacks). Let us consider how weighted resolution algorithms work with the Whitney umbrella. The singularity invariant of \cite{ATW-weighted} is $(2,2)$ at the normal crossings points and $(2,3,3)$ at the pinch point. The weighted algorithm improves the invariant at each step, hence it must get rid of the pinch point, and, indeed, the first blow up is associated to the center $\cJ=(x^2,y^3,z^3)$ supported on the pinch point and $\cX'=\Bl_{\overline \cJ}(X)$ is a normal crossings orbifold with the maximal invariant equal to $(2,2)$.

Let us show the relevant computation. Once again, the $z$-chart is the important one because it contains $p'$. This time the new coordinates on the corresponding \'etale-local chart $Y'_z=\Spec(k[x',y',z'])$ are related to the old ones by $x'=\frac{x}{z'^3}$, $y'=\frac{y}{z'^2}$ and $z=z'^2$, so the pullback of $X$ is given in  by $z'^6(x'^2-y'^2)$, and the $z$-chart of the strict transform $X'_z$ is given by $x'^2-y'^2=0$. In particular, the strict transform is indeed a normal crossings surface.

\subsubsection{The stack structure}\label{Sec:stacky-structure}
Finally, let us also describe the stack structure. The $z$-chart stack is obtained by dividing the \'etale chart by the $\bmu_2$ action, which sends $(x',y',z')$ to $(-x',y',-z')$. Equivalently, $x',y',z'$ are homogeneous of weights $1,0,1$ in $\ZZ/2\ZZ$.\footnote{This formulation holds also in characteristic 2, but ill suited to the singularities, as $\bmu_2$-torsors are not smooth.} The  $z$-chart stack of the ambient space is $\cY'_z=[Y'_z/\bmu_2]$, so its coarse space is $\underline{\cY'_z} = \Spec(k[x'^2,x'z',z'^2,y'])$. The nontrivial stack structure with inertia $\bmu_2$ is concentrated on the $y'$-axis $V(x',z')$, which describes the singular locus of $\underline{\cY'_z}$, providing a stack theoretic resolution of the coarse space.

The $z$-chart of the weighted blowing up of $X$ is the closed substack $\cX'_z=[X'_z/\bmu_2]$ of $\cY'_z$. It meets the $y'$-axis, where the inertia jumps to $\bmu_2$ at the point $x'=y'=z'=0$. The coarse space is easily seen to be equal to $X'_z/\bmu_2=\Spec(k[x'',y',z]/(x''^2-y'^2z))$, where $x''=x'z'=\frac{x}{z}$. It is embedded in $\underline{\cY'_z} = \Spec(k[x'^2,x'z',z'^2,y']) = \Spec(k[s,x'',z,y'])/(sz-x''^2)$ by setting $s = y'^2$. Thus, the coarse space is (the $z$-chart of) the usual Whitney umbrella obtained by blowing up the classical center $(x,y,z)$. The nontrivial stack structure along the ambient $y'$-axis inserts a stack structure at the origin $x''=y'=z=0$ and resolves it to a normal crossings orbifold.

\subsubsection{An alternative modification}\label{Sec:stacky-structure-alternate}
\label{stackyrem}
One can also provide a normal crossings stack modification of $X$ using the weighted blow up associated to the center $(x^{1/2},z)$ supported along the $y$ axis.  On the ambient space $Y$, the $y$-axis is replaced by a $(1,2)$-weighted $\PP^1$-bundle with $\bmu_2$ stabilizers along the intersection of the exceptional divisor with the $yz$-plane. The proper transform of the surface $X$ meets this locus at $y=0$. It is precisely the stack of Theorem \ref{Th:wild-pp}\eqref{It:wild-pp-stack}, whose coarse moduli space is $X$.

\subsubsection{Resolving the cover}
Now, we can conceptually explain how an unsolvable problem became solvable in the orbifold world, circumventing the obstruction met in Koll\'ar's proof. Let us focus  on the curve $C = \Spec k[z]$. In either the procedure of Section \ref{Sec:stacky-structure} or \ref{Sec:stacky-structure-alternate} it is replaced by the root stack $C'=\left[\Spec k[z^{1/2}] \ /\ \bmu_2\right]$  at $z=0$. The ramified double cover $\tilde C \to C$, where  $\tilde C =\Spec k[z^{1/2}]$, now factors through an \emph{unramified} double cover $\tilde C \to C'$.


To summarize, the miracle of \emph{weighted} NC-preserving resolution is possible because the stack structure of $X'$ induces the stack structure $C'$ on $C$, which is precisely the stack structure  induced from the branched cover $\tilC\to C$.

\subsection{Inseparable umbrellas in characteristic 2}\label{Sec:inseparable-umbrellas}
In this section we assume that $\chara(k)=2$ and study analogues of the Whitney umbrella. Due to the characteristic assumption there are no minus signs in this section.

\subsubsection{Inseparable umbrellas and their invariants}
The same formula $X=V(x^2+y^2z)$ defines a surface whose singular locus $C=V(x,y)$ is not generically normal crossings. In fact, $X$ is equisingular along $C$ because there exists a $\GG_a$ action on $X$ which induces the usual additive action on $C$. It is defined by $a(x,y,z)=(x+ay,y,z+a^2)$.
In a sense the whole $C$ consists of inseparable pinch points.

This example is due to W{\l}odarczyk \cite{Wlodarczyk-cobordant}, and it was used to conclude that a weighted resolution algorithm should re-define the invariant. How this might be done remains a challenge:

If $k$ is perfect, and one defines the weighted invariant  (or multiorder, \cite{Temkin-weighted-excellent}) of an ideal similarly to the characteristic zero case of \cite{ATW-weighted} in terms of regular parameters at a given point, then the \emph{generic} point of $C$ has invariant $(2,2)$ and any closed point has invariant $(2,3,3)$, so the invariant is not upper semicontinuous.

If instead one defines the weighted invariant at the \emph{geometric} generic point $\Spec \overline{k(z)} \to C \subset X$, interpreted  as a \emph{closed} point $(0,0,z)$ of the base change $X_{ \overline{k(z)}}$, the invariant remains $(2,3,3)$, since $\overline{k(z)} $ is perfect. This approach is consistent with the interpretation of the invariant in \cite{Temkin-weighted-excellent} in terms of embeddings of \emph{tube schemes}, an approach that is inherently stable under such base change and makes the invariant upper-semicontinuous on \emph{geometric} points.
The maximal locus here is the $z$-axis, which is indeed a good blow-up center.

{This however does not lead to a resolution algorithm in general.}
 The non-reduced truncation $V(x^2+y^2z, z^{100})$ is an example of a scheme where the maximal locus is even non-reduced. {The wild umbrellas with $n>1$ are reduced examples with non-reduced maximal locus.}  Wlodarczyk gives examples with maximal locus having arbitrary singularities.



One can also imagine the possibility that a version of a basic weighted algorithm just does not exist in positive characteristic, and a resolution invariant should be of a more complicated nature, taking inseparability and wild ramification issues into account.


\subsubsection{Behavior of the inseparable umbrella under weighted blowups}\label{Sec:inseparable-blowup} Another phenomenon worth notice is the behavior of the singularity under the natural weighted blowup associated to the center $(x^2,y^3,z^3)$ at the origin.

Writing the equations $x = s^3x', y=s^2y', z = s^2z'$ of the degeneration to the normal cone, the proper transform has equation
$x'^2 = y'^2z'$. We learn three lessons from this:
\begin{enumerate}
 \item The singularity invariant does not increase: singularities on the weighted normal cone are no worse than the original singularities. This is automatic for weighted homogeneous singularities; but
 \item unlike the characteristic 0 case, the invariant does not drop on blowing up. After all, the same invariant is taken along all the $z$ axis, which is not modified.
 \item\label{Item:mu2charts} Finally, unlike in characteristic 0, there are no \'etale charts to use here, and \emph{flat charts which are not smooth are useless.}

 The flat chart given by the slice $\{z'=1\}$, where the equation is $x'^2 = y'^2$ with $\bmu_2$ acting with weights $1$ on $x'$ and $0$ on $y'$, is misleading: the singularity of a $\bmu_2$-torsor $T$ over a stack $Z$ is typically worse than that of $Z$.
 In this case the invariant $(2)$ of $T := \{x'^2 = y'^2\}$ is larger than the invariant $(2,3,3)$ of $Z:=[\{x'^2 = y'^2z'\} / \GG_m]$.
 \end{enumerate}

\subsection{Wild umbrellas in characteristic 2}\label{Sec:wild-umbrellas}
\subsubsection{Wild quadratic extensions in characteristic two}\label{quadsec}
Before going further let us briefly recall a few basic facts about wild quadratic extensions. To give a broader perspective, a more detailed and illuminating discussion of basic facts (in any characteristic) is given in the appendix, but here we consider only a quadratic extension $K'/K$ of discrete valued fields of characteristic zero with rings of integers $\cO'/\cO$. Also, we assume that $e_{K'/K}=2$. In this case, the different $\delta=\delta_{K'/K}$ of the extension is a positive even number and $\cO'=\cO[\pi']$, where $\pi'$ is a uniformizer with the minimal polynomial of the form $f_{\pi'}=t^2+u\pi^{\delta/2}t+\pi$ with a uniformizer $\pi\in\cO$ and a unit $u\in\cO^\times$.

Moreover, if the residue field $k=\cO/m_\cO$ is algebraically closed, then up to completion and automorphisms of both $K$ and $K'$ the extension is determined by the different, and one can take $u=1$. Namely, $\hatcO=k\llbracket\pi\rrbracket$, $\hatcO'=k\llbracket\pi'\rrbracket$ and $\pi'^2+\pi^{\delta/2}\pi'+\pi=0$ for an appropriate choice of uniformizers. Alternatively, one can choose uniformizers of $\hatK$ and $\hatK'$ so that $\pi=\pi'^2+\pi'^{\delta+1}$.

\subsubsection{The wild umbrellas}
Now we define a sequence of wild umbrellas by  formula \eqref{Eq:pp(n)} (page \pageref{Eq:pp(n)}): $$X_n=V(x^2+y^2z+xyz^n).$$ To simplify notation fix $n$ and set $X=X_n$. As earlier, the singular locus is the $z$-axis $C=V(x,y)$, but this time the quadratic form is non-degenerate at any point of $C_0=C\setminus\{p\}$, so similarly to the usual Whitney umbrella there is a pinch point at the origin and all other singularities are normal crossings singularities.

\subsubsection{Persistence of the singularity}\label{perssec}
The equations $f(x,y,z) = x^2+y^2z+xyz^n$ are homogeneous of degree 2 in $x,y$, hence the same computation as with the classical Whitney umbrella shows that after blowing up the origin the strict transform in the $z$-chart is given by the same polynomial $f(x',y',z)=x'^2+y'^2z+x'y'z^n$. Of course the blowup of the singular locus $V(x,y)$ resolves the singularities, but is not  NC-preserving.

Let us now describe the weighted blowing up associated to the center $(x^2,y^3,z^3)$, 
which is easily seen to be the weighted $\cI_X$-admissible center of maximal possible multiorder. As explained in  \ref{Sec:inseparable-blowup}\eqref{Item:mu2charts} above, one must not use $\bmu_2$-charts as they fail to describe singularities faithfully. We use instead the degeneration to the normal cone $x = s^3x', y = s^2y', z=s^2z'$. The proper transform of $f(x,y,z)$ is
\begin{equation} \label{Eq:wild-cone} x'^2+y'^2z'+s^{2n-1}x'y'z'^n.\end{equation}


{These are equisingular along $V(s,x',y')$, but this locus does not support a center of weighted blowup, similar to the inseparable umbrella.}

\subsubsection{Umbrellas as coarse moduli spaces}\label{Sec:wild-pp-stack} Consider a general pinch point \emph{in arbitrary characteristic} --- a singularity $p \in C \in X$, where $C$ is a smooth curve, whose preimage on the normalization $X^\nor \to X$ is a double cover $\tC \to C$ by a smooth curve with branch locus $p\in C$. Let $\sigma\in \Aut(\tC)$ be the involution associated to the double cover.

Let $Y = X^\nor \cup_\tC X^\nor$, where the gluing map is given by $\sigma$. Its normalization is $\tY = X^\nor \sqcup X^\nor$, on which we have a natural involution $\tilde \tau$ switching the components. This descends to an involution $\tau$ on $Y$, which restricts to $\sigma$ on $\tC$.

\begin{proof}[Proof of Theorem \ref{Th:wild-pp} Part \ref{It:wild-pp-stack}]
Since pinch points are isolated, it suffices to consider one pinch point at a time, and glue things together.

Note that $X = Y/(\ZZ/2\ZZ)$, and the action is free away from the point $p$. Taking the stack quotient $\cX = [Y\,/\,(\ZZ/2\ZZ)]$, we have that $\cX \to X$ is the morphism to the coarse moduli space, an isomorphism away from $p$. Finally $Y$ is a normal crossings surface by construction.
\end{proof}

\subsubsection{Wild umbrellas and weighted blowups}\label{Sec:wild-pp-no-tame}  Specializing again to characteristic 2, the cover $\tC \to C$ is necessarily wild at the point $p$.

Theorem \ref{Th:wild-pp} Part \ref{It:wild-pp-no-tame} now follows from:
\begin{proposition}\label{tameunresolve}
Let $X$ be a wild Whitney umbrella. Then for any stack theoretic modification $f\:X'\to X$ such that $X'$ is a tame stack and $f$ is an isomorphism over $X\setminus\{p\}$, the stack $X'$ is not normal crossings.
\end{proposition}
\begin{proof}
Let $C'\into X'$ be the strict transform of $C$. Then $C'$ is a tame stack obtained from $C$ by inserting a stack structure at the origin.

\begin{claim} The morphism  $C' \to C$ induces a homeomorphism of \'etale topologies, in particular an isomorphism on \'etale fundamental groups.\end{claim}

Explicitly, by \cite[Theorem 3.2]{AOV1}, locally in the \'etale topology of $C$, the stack $C'$ is a quotient stack, necessarily by some $\bmu_m$. Such a stack is a $\bmu_m$-root stack along $p$. It can be written explicitly as $\left[\,\left(\Spec\, \cO_C[\eta]\,/\,(\eta^m-z)\right)\ \big/ \ \bmu_m\,\right]$. \qed


 If $X'$ is normal crossings, then the normalization induces an \emph{\'etale} double cover $\tilC'\to C'$, corresponding to a homomorphism $\phi'$ of the \'etale fundamental group of $C'$ to $\ZZ/2\ZZ$.

We have a \emph{commutative} diagram
{
$$
\xymatrix{\tilde \phi' \ar@{}[r]|-<\in
\ar@{|->}[d]
&\quad\Hom(\pi_1(C'),\ZZ/2\ZZ) \ar[rr]^{\simeq}_{\text{(Claim)}} \ar[d]_{\operatorname{res}'} &&  \Hom(\pi_1(C),\ZZ/2\ZZ) \ar[d]^{\operatorname{res}}
&\ar@{}[l]|-<{\ni} {\pmb\nexists} \tilde \phi\hphantom{\pmb\nexists} 
 \ar@{|.>}[d]
\\
\phi'
\ar@{}[r]|-<\in
 &Hom(\pi_1(C'\smallsetminus \{p\} ),\ZZ/2\ZZ) \ar@{=}[rr] && Hom(\pi_1(C\smallsetminus \{p\}),\ZZ/2\ZZ)&\ar@{}[l]|-<\ni \phi.}$$
}
The \'etale 2-cover $\tC \setminus \{p\} \to C\setminus \{p\}$ corresponds to an element $$\phi \in  Hom(\pi_1(C\smallsetminus \{p\}),\ZZ/2\ZZ).$$  Since the cover is wild, hence not \'etale, it is not the restriction of an element  $ \tilde \phi \in \Hom(\pi_1(C),\ZZ/2\ZZ)$.

On the other hand the corresponding cover $\tilC'\setminus \{p\} \to C'\setminus \{p\}$ does extend to the \'etale double cover  $\tilC'\to C'$, corresponding to $\tilde\phi' \in \Hom(\pi_1(C'),\ZZ/2\ZZ)$. Its image in $\Hom(\pi_1(C),\ZZ/2\ZZ)$ restricts to $\phi$ since the diagram is commutative, a contradiction.
%
\end{proof}

\subsubsection{Wild pinch points and their normalization --- explicit study}\label{Sec:wild-explicit} We consider the collection of singular surfaces $X=X_n \subset \AA^3$ in characteristic 2 given by $$x^2 + xyz^n + y^2z=0.$$
These surfaces are not normal. Normalization is given by $w = x/y$ so that $$w^2 + wz^n + z,$$ a regular surface with conductor given by $y=0$.   Once again monodromy at the origin is $\ZZ/2\ZZ$, and we have seen in Proposition \ref{tameunresolve} that no weighted blowup will trivialize it, since weighted blowups are  tame at heart.

\subsubsection{Covers} As shown in Section \ref{Sec:wild-pp-stack}, this scheme does admit a ``partial resolution'' preserving normal crossings, by a Deligne--Mumford stack which is not tame. Let us describe it explicitly.

Consider $$T = \Spec k[w',y_1,y_2,z] / (w'^2 + z^nw' + z, y_1y_2)$$ with the $\ZZ/2\ZZ$-action $$(w',y_1,y_2,z) \mapsto (w'+z^n, y_2,y_1,z).$$

Then we have an invariant morphism $T \to X$ defined by $$(x,y,z)\quad  \mapsto\quad  (w'(y_1+y_2) +z^n y_2 ,\ y_1+y_2, \ z).$$
Also $X' := [T/(\ZZ/2\ZZ)] \to X$ is the coarse moduli space  morphism; it is easily seen to be bijective on isomorphism classes of geometric points, and properness is the delicate point. But on each component $T_i = V(y_i)$ the map $T_i \to X$ factors through an isomorphism with the normalization $\tilde X$, so $T\to X$ is finite.

As a sanity check note that
\begin{equation} \label{Eq:sanity} x^2 + xyz^n + y^2z = w'^2 y^2 + z^{2n} y_2^2 + (w'y+ z^ny_2)yz^n + y^2z = z^{2n} y_1y_2 = 0,\end{equation}
using $y_1y_2=0$ and  $w'^2 + z^nw' = z.$




To tie this to the notation of Section \ref{Sec:wild-explicit} above, on the normalization $\tilde T$ of $T$ one has $w'=w$ on the locus $y_2=0$ and $w' = w+z^n$ on the locus $y_1=0$.

\subsubsection{Transformation on the ambient space: a remaining challenge}\label{Sec:ambient-challenge}
A remaining challenge is to define  a  \emph{natural, regular, stack-theoretic} transformation of the ambient space $\AA^3$ which gives rise, on proper transforms, to the morphism $T \to X$, or to a related stack-theoretic resolution of $X$. Write $$C_{w',z} = V(w'^2 + z^nw' + z)\subset \Spec k[w',z]\qquad \text{ and }\qquad \AA^2_{y_1,y_2} = \Spec k[y_1,y_2].$$

The displayed equation \eqref{Eq:sanity} above
$$x^2 + xyz^n + y^2z = z^{2n} y_1y_2 $$
allows one to extend the $\ZZ/2\ZZ$-morphism $T \to X$ to a $\ZZ/2\ZZ$-morphism $C_{w',z}\times  \AA^2_{y_1,y_2} \to P$, where $P$ is the blowup of $\AA^3_{x,y,z}$ along the ideal $ (x^2 + xyz^n + y^2z, z^{2n})$. Setting $u=y_1y_2$, the morphism lands in the chart $x^2 + xyz^n + y^2z= z^{2n}\cdot u$ corresponding to $z^{2n}$.  The morphism is finite of degree 2 onto this chart  and provides a stack resolution of the chart.

The other chart $(x^2 + xyz^n + y^2z)\cdot v= z^{2n} $ is singular where $\{z=x=yv=0\}$. A compatible resolution here would provide an embedded wild stacky normal crossings resolution of $S \subset \AA^3$. This would give an answer of sorts to Problem \ref{Prob:embedded}, though one would like it to arise from a general procedure, analogous to the wild root-stack construction of \cite{Kobin}.

\section{Characteristic 0}\label{Sec:nc0}

\subsection{Stack-theoretic NC-preserving  resolution in general}\label{Sec:nc}

Our goal here is to describe a principle that guarantees  results such as Theorem    \ref{Th:nc} to hold. {We also prove that the principle holds for normal-crossings points.}

As explained in \cite[Section 1.2]{BDS-B-RL}, one reduces, using classical resolution methods, to the case where $X\subset Y$ is a divisor in a smooth variety $Y$. Assuming that $X$ is a divisor  simplifies the discussion somewhat.\footnote{We note that  \cite{Wlodarczyk-nc} avoids this step, at the small price of discussing invariants in higher codimension.}

We denote by $X^{\nc(k)}$ the locus of normal crossings points of order  $k$. Such points are called \emph{$\nc(k)$-points.}

\subsubsection{Wonderful invariants} Recall that the invariant $\inv_{\cI_X}$ introduced in \cite{ATW-weighted} of a normal crossings divisor at an $\nc(k)$-point is $(k)_k:=(k,\ldots,k)$ with $k$ entries. For the present paragraph we   only need the following, see \cite[Section 3.3.33]{Wlodarczyk-cobordant}:
\begin{enumerate}
\item \emph{$\inv:|X| \to \Gamma$ is a singularity invariant:} an upper semicontinuous function taking values in a well ordered set, and taking the minimum value precisely on smooth points of $X$.
\item \emph{$\inv$ is functorial} for smooth morphisms (or analytic submersions).
\item \emph{$\inv$ has smooth level sets:} the sets $W_K = \{p\in X: \inv_X(p) = K\}$ are smooth.
\item \emph{$\inv$ has immediately reducible level sets:} the set $W_K$ carries a weighted center of blowup $J_K$; the corresponding blowup of the open set  $\{p\in X: \inv_X(p) \leq K\}$ has maximal invariant $<K$.
\end{enumerate}

We say such $\inv$ is a \emph{wonderful invariant}. Apart from \cite{ATW-weighted}, such wonderful  invariants  were introduced for the normal crossings context in \cite{Wlodarczyk-cobordant, ABQTW} and in the foliated case in \cite{ABTW-foliated}.

\subsubsection{Adequate classes of singularities.} We rely on two properties of the class $\cC$ of hypersurface normal crossings singularities $\{(p\in X)\}$ in the presence of a wonderful invariant:

\begin{definition} \hfill
\begin{enumerate}
\item\label{It:open-support} A class $\cC$ of singularities   \emph{has  open support} if the locus of  points $\{p\in X|(p\in X) \in \cC\}$ is open.
\item\label{It:inv-closed}  The class is \emph{$\inv$-closed} if  the locus  of points where $(p\in X) \in \cC$ and $\inv_{\cI_X}(p) = K$ is \emph{constant} is closed in the level set $W_K = \{p\in X: \inv_X(p) = K\}$.
\end{enumerate}
If both conditions hold we say $\cC$ is \emph{adequate}.
\end{definition}

\begin{theorem}\label{Th:principle} Let\ \  $\inv$ be a  wonderful invariant, and let $\cC$ be an adequate class of singularities. Then any variety $X$ admits a $\cC$-preserving weighted  resolution, namely a modification $X'\to X$, where $X'$ has singularities in $\cC$ and $X' \to X$ is an isomorphism over the locus of $\cC$ points $\{x\in X: (x\in X) \in\cC\}$.
\end{theorem}
\begin{proof} Consider the the function $$\inv^\cC(p) :=\begin{cases} \min \Gamma & x\in X^\cC \\ \inv(p) & x\notin X^\cC.\end{cases}$$ It is the modified invariant that asks us to pretend that points in $\cC$ are smooth. Since $X^\cC\subset X$ is open, the function $\inv^\cC$ is upper semicontinuous.

If $X \neq X^\cC$ and $K= \max_{p\in X}\inv^\cC(p)$ the maximal value,   the maximal locus of $\inv^\cC$ is $W_K^{\notin\cC} :=W_K \setminus X^\cC$. Note that the set $\{p\in X^\cC: \inv(p)\geq K\}$ is closed in $X$ and disjoint from $W_K^{\notin\cC}$. Therefore $W_K^{\notin\cC}$  has an open neighborhood $X_\circ := X \setminus \{p\in X^\cC: \inv(p)\geq K\}$.

By assumption $W_K$ carries an $\inv$-reducing weighted center of blowup $J_K$. Its restriction to $X_\circ$ reduces the invariant on this set by functoriality.
Therefore the weighted center of blowup $$J_K^\cC = \begin{cases} J_K & p\in X_\circ \\
(1) & p\notin X_\circ\end{cases}$$
is an $\inv^\cC$-reducing invariant everywhere: if $X' \to X$ is its blowup, then $\max_{p'\in X'} \inv^\cC(p')<K$, as needed.





\end{proof}

\subsubsection{Normal crossings singularities}
Normal crossings hypersurface singularities have open support by definition. Our next goal is to prove that the class is $\inv$-closed, in particular reducing Theorem \ref{Th:nc} to Theorem \ref{Th:principle}:
\begin{proposition}\label{Lem:nc(k)}
Let $X\subset Y$ be a hypersurface in a smooth variety over a field $F$ of characteristic 0. Set $K= (k)_k$. Then $X^{\nc(k)}$ is closed in the level set $W_K$.
\end{proposition}

We start with a simple lemma that sheds light on the issue at hand.

\begin{lemma}\label{ncinv}
Let $X\subset Y$ be a hypersurface in a smooth variety in characteristic 0 and $x\in X$ a point. Let $k$ be the number of branches of $X$ at $x$ and set $K=(k,\ldots,k)=(k)_k$. Then $\inv_X(x)\ge K$ and the equality holds if and only if $X$ is normal-crossings at $x$.
\end{lemma}
\begin{proof}
The claim is \'etale-local, hence we can separate branches to irreducible components and assume that $X=V(\cI)$, where $\cI=(f_1\ldots f_k)$ with $f_1\.f_k\in m_x$. It follows that $\cD_Y^{\le (k-1)}(\cI)\subseteq(f_1\.f_k)$ and hence $\ord_x(\cI)\ge k$. Thus, $\inv_X(x)$ starts with an entry $\geq k$ and hence $\inv_X(x)\le K$ if and only if there exist at least $k$ linearly independent maximal contacts, that is, the dimension $d$ of the image of $\cD_Y^{\le (k-1)}$ in $m_x/m_x^2$ is at least $k$. If this happens then the images of $f_1\.f_k$ in $m_x/m_x^2$ are linearly independent, and hence $X$ is nc at $x$ and the exact equality $\inv_X(x)=K$ holds.
\end{proof}

Unfortunately, the number of branches is not upper semicontinuous, as the example of the Whitney umbrella $V(x^2-y^2z)$ shows: here the general singularity point $\eta$ is normal-crossings with invariant $(2,2)$ with two  transversal branches, but the pinch point $w$ is unibranch and has invariant $(2,3,3)$.  This is reflected concretely by the observation that $\sqrt{z}$ lies in the strict henselization $\cO_\eta^{sh}$ and generates an extension which separates the branches and makes the singularity simple normal-crossings, but the extension to $w$ is ramified and $\sqrt{z}\notin\cO_w^{sh}$. So, Proposition \ref{Lem:nc(k)} is not immediate, and all the arguments we could find are a bit subtle.

\begin{proof}[Proof of Proposition \ref{Lem:nc(k)}]

Assume that $w\in W_K$ possesses a generization $\eta\in W_K$ which is normal-crossings in $X$. We should prove that $w$ is also normal-crossings at $w$.

\emph{We start with a few  reductions.} We can freely replace $X\subset Y$ by taking an \'etale neighborhood of $w$ and perform finite base field extensions during the proof. In addition, replacing $w$ by a closed point to which it specializes in $W_K$ and replacing $F$ by a finite extension, we can assume that $w$ is an $F$-point. Replacing $\eta$ by the generic point of the component of $W_K$ in which it lies preserves both conditions, hence we can assume that $\eta$ is a generic point of a component of $W_K$. In particular, we see that $W_K$ is of codimension $k$.\footnote{This is true even without the normal-crossings assumption} Since the invariant is upper semicontinuous we can remove all points with $\inv_X>K$, so that $W=W_K$ is now the maximality locus of $\inv_X$, and hence underlies a closed weighted center. In particular, $W$ is smooth and we can choose local parameters $\ut=(t_1\.t_k),\uy=(y_1\.y_{n-k})$ of $\cO_{Y,w}$ so that $\cI_W=(t_1\.t_k)$ is the ideal of $W$ on $Y$. The question being local, we may remove all components of $W_K$ except the one containing $\eta $ and $w$, so that $W$ is geometrically irreducible.

\emph{Leading terms and normal cones.} Let $X=V(f)$ with $f\in\cO_{Y,w}$. Since $\inv_X(w)=K$, the element $f$ is of order $k$ at $w$ and has $k$ maximal contacts vanishing along $W$ and linearly independent modulo $\cI_W^2$, that is, $\cD_Y^{\le (k-1)}(f)=\cI_W$.\footnote{In fact, $\inv_X(w)=K$ along $W$, regardless of normal-crossings, if and only if $\cD_Y^{\le (k-1)}(\cI_X)=\cI_W$.}

The image $f_k$ of $f$ in the graded coordinate ring $\oplus_{n\geq 0} \cI_W^n/\cI_W^{n+1}$ of the normal bundle $N_W\to W$ of $W$ in $Y$ describes the normal cone $C$ of $W$ in $X$.  We denote its fiber over $w$ by $C_w$. It is defined in the vector space $N_w$ by the homogeneous form $(f_k \mod m_{w,W})$.



Looking at the expansion of $f$ in $\hatcO_w=F\llbracket \ut,\uy\rrbracket$ one obtains that any monomial of $f$ is of degree at least $k$ in $\ut$ as otherwise there would exist an element in $\cD_Y^{\le (k-1)}(f)$ not contained in $\cI$. Separating terms depending on $\ut$ only, $f=f_k^0+g$, where $g\in\cI_W^{k}$ and $f_k^0\in F[\ut]$ is homogeneous of degree $k$ and independent of $\uy$ --- it is the lift of $(f_k \mod m_{w,W})$ obtained from separating the variables. Restricting the differential expressions of the maximal contact elements of $X$ onto the $k$-truncations we see that $\cD_Y^{\le (k-1)}(f_k^0)$ contains $k$ elements linearly independent modulo $\cI_W^2$, and hence $\cD_Y^{\le (k-1)}(f_k^0)$ and $V(f_k^0)$ also has invariant $K$ along $W$. \footnote{In fact, we just recovered in our setting a much more general fact, that the fiber at $w$ of the weighted normal cone $C$ to any $Z$ in $Y$ has the same invariant as $Z$. In our case the weights are $(k\.k)$, so this is the usual tangent cone to $X$.}

\emph{We show that $C_w$ is normal crossings.}
Now, it is  time to use the normal-crossings condition at $\eta$ for the first time. Set $E=k(\eta)$. Since $X$ is normal-crossings at $\eta$, so is $C= V(f_k)$. Thus there exists a finite extension $E'/E$ such that $f_k=\prod_{i=1}^k\ell_i$ in $E'\llbracket\ut\rrbracket$ and this decomposition restricts, by setting $\uy=0$,  to a decomposition of $f_k^0$ into linear factors $\ell_i^0$. However, $f_k^0$ has coefficients in $F$ (rather than $F[\uy]$), and hence the decomposition of $f_k^0$ has coefficients in $F'=E'\cap\overline{F}$ which is a finite extension of $F$. Thus, replacing $F$ by a finite extension we can assume that $f_k^0$ is a product of $k$ linear forms. Since the invariant of $C_w$ is $K$, these are linearly independent linear terms.

\emph{We show that the cone $C$ of $W$ in $X$ is normal crossings.} Consider the dual projective space $\PP N_W^\vee$ parametrizing hyperplanes in the fibers of $\PP N_W \to W$. It is proper over $W$. The condition on a hyperplane being contained in the projective cone $\PP C$ of $W$ in $X$ is closed, so defines a closed subscheme $H_X \subset \PP N_W^\vee$, hence the space $H_X$ of hyperplanes contained in $\PP C$ is proper over $W$. Its generic fiber is finite and reduced of degree $k$, since $C_\eta$ is normal crossings by assumption. The argument above shows that  the fiber of $H_X$ over $w$ has precisely $k$ reduced points. It follows that $H_X$ is finite \'etale over $w\in W$ of degree $k$.

Replacing $Y$ by an \'etale neighborhood of $w$ we may ensure that $H_X$ becomes a disjoint union of $k$ copies of $W$. The universal homogeneous linear form on $N$ defining the corresponding family of hyperplanes defines elements $t_i$ so that, on the chosen \'etale neighborhood, we have $f_k = t_1\cdots t_k$, and by Lemma \ref{ncinv} it has normal crossings, as needed.\footnote{In fact the lemma is unnecessary here --- the forms are linear therefore have order 1.}

 \emph{Having thus replaced $Y$ by an \'etale neighborhood, we have that $C$ is \emph{simple} normal-crossings.}
Unfortunately, the example of $f=x(yz+x^3)=xyz+x^4$ shows that even when $\inv_X(w)=K$ and the tangent cone is simple normal-crossings, $X$ does not have to be normal-crossings at $w$ --- the factorization of the $k$-form does not have to lift further. All we have is that $f = t_1\ldots t_k+ g$ where $g\in (\ut)^{k+1}$. So, we will have to use that $X$ is normal-crossings at $\eta$ once again, and show that the branches at $\eta$ persist at $w$.

\emph{We apply induction on $k$, so that we may assume that the proposition holds for $nc(k')$ points, for all $X$ and all $k'<k$}.

\emph{We show that  the blowup of $X' \to X$ along $W$ is normal crossings transversal to the exceptional.} Using the formalism of the degeneration to the normal cone \cite{Wlodarczyk-cobordant, Quek-Rydh,ABTW-foliated}, the blowup has equation of the form $f' = t'_1\cdots  t'_k+sh(s,t',y)$, where $s$ is the exceptional variable and $t_i'$ are the transformed variables satisfying $t_i = st_i'$ for $i=1\ldots,k$.

Consider a point $p_w$ on the exceptional $E_W=\{s=0\}$ above $w$ where $t_1'=\cdots =t_{k'}'=0$ but $t'_{k'+1},\ldots,t'_{k}$ are units. On the one hand, such a point is a specialization of the point $p_\eta$ with the same coordinates $t'_i$ over $\eta$, which is an $\nc(k')$ point with invariant $K'=(k',\ldots,k')$. By upper semicontinuity $\inv_{X'}(p_w) \geq K'$.

On the other hand, we claim that the invariant at $p_w$ of $X' = V(f')$ is $\leq K'$. Indeed, taking the \emph{logarithmic} invariant of \cite{Quek, ABTW-foliated} with respect to the divisor $V(s)$ boils down to the invariant of the restriction $t'_1\cdots  t'_k$ of $f'$ to $\{s=0\}$. This restriction is the projectivised normal cone $\PP C$,  which has invariant $K'$ at such point, since it is normal crossings.  One therefore obtains $K' = \loginv_{X'}(p_w) \geq \inv_{X'}(p_w)$.\footnote{Alternatively, this follows from upper semicontinuity of $\inv$ in families, considering $s$ as a parameter.}

It follows that $\inv_{X'}(p_w) = K'$, and, being a specialization of an $\nc(k')$-point, the induction assumption implies it is an $\nc(k')$-point. Thus $X'$ is normal crossings along $E_W$.

Moreover,  the normal crossings branches of $X'$ meet $E_W$ transversally along $t_i'=0$. Indeed, in the equation above, taking the derivative $\partial_{t_2'}\cdots \partial_{t_{k'}'}(f)$ and clearing units defines at $p_w$ a maximal contact element of the form $t_1'+ sh_1(s,\ut',\uy)$. Permuting the variables  one obtains $k'$ maximal contact elements of the form $t_j'+ sh_j(s,\ut',\uy)$ at $p_w$.
This forces the normal crossings factors at $p_w$ to be of the same form, as in the proof of Lemma \ref{ncinv}.


\emph{We conclude that $X$ is normal crossings at $w$.} Passing to an open neighborhood of $E_W$ we may assume $X'$ is normal crossings, hence its normalization $\tX' \to X'$ is regular. The normalization $\tX'$ contains the $k$ \emph{disjoint}
loci $E_W^i := E_W\cap v(t_i')$, whose union is the preimage of $W$. Since these loci are disjoint, an \'etale neighborhood of each one maps to  a distinct branch of $X$ along $W$.\footnote{It also follows from Artin's contractibility criterion \cite{Artin} that $E_W^i$ are contracted to disjoint loci isomorphic to $W$ in the normalization $\tX\to X$ of $X$. This contractibility also follows from the more concrete but less general Castelnuovo-Moishezon's contraction theorem  \cite{Moishezon}, \cite[Theorem 1.1]{arenacontraction}: these loci are regular, each of which is a hyperplane in $E_W$ isomorphic to $\PP^{k-2}_W$, in particular its normal bundle is $\cO(-1)$ along the fibers.} 

It follows that $X$ has $k$  branches at $w$, and by Lemma  \ref{ncinv}  $X$ is normal crossings at $w$.
\end{proof}

\subsubsection{Where this comes from.} We were working on the present manuscript, focusing on the previous sections, when we happened upon an argument for Theorem \ref{Th:nc}. When we shared a draft of this argument with Belotto da Silva and with W{\l}odarczyk, we learned of their soon-to-be-posted papers  \cite{BDS-B-nc,Wlodarczyk-nc}. We therefore repurposed this section to bring out the underlying principles. The core of this section was posted as \cite{AT-partial}. This version expands the treatment and corrects errors in earlier versions.

\subsubsection{Where this might go.} 


Mathematically, one might seek other adequate classes of singularities of interest, though low-hanging fruit is already picked:

Another adequate class of hypersurface singularities  is the class $\{\inv\leq K\}$ of singularities $(p\in X)$ with $\inv(p)\leq K$, but this is a triviality.  Jaros{\l}aw W{\l}odarczyk proves Theorem \ref{Th:nc} in \cite[Theorem 2.0.3]{Wlodarczyk-nc}  for the wider class of non-reduced normal crossings singularities, \'etale locally given by $\prod x_i^{k_i}$ for a sequence of local parameters $x_i$. One can check that the argument of Proposition  \ref{Lem:nc(k)} applies here too, as can be deduced from \cite[Theorem 1.1.1]{Wlodarczyk-nc}.

One can similarly apply Theorem \ref{Th:principle} to normal crossings singularities in the presence of normal crossings exceptional divisors, using logarithmic invariants \cite{Wlodarczyk-cobordant, ABQTW}.
This would likely not go beyond  \cite{BDS-B-nc,Wlodarczyk-nc}.     The question has not been addressed in the foliated case of \cite{ABQTW}.

A non-example is the class $\cC$ of configurations of hyperplanes, which is not $\inv$-closed: the variety $V((x^2-t^2y^2)y)$ has maximal locus $W_{(3,3)} = V(x,y)$, which includes the origin, the only point not in $\cC$. We are informed by W{\l}odarczyk that a student is studying a variant that is $\inv$-closed.

\subsection{Destackification} \label{Sec:poly-circulant}

Consider now a separated normal-crossings DM stack $X$ with coarse moduli space $\oX$.

Working \'etale-locally on $X$ we may assume a codimension-1 closed embedding in a smooth stack $X \subset Y$ exists.  Indeed, working \'etale locally over $\oX$, we may assume $X = [V/G]$, where $V$ is affine and $G$ is the stabilizer of a point $p$ of $X$. Lifting generators of $\cO(V)$ we obtain a $G$-equivariant embedding $V \subset W$ with $W$ affine and regular. Write $\cI_V$ for the ideal. The elements of $\cI_{V,p} / m_p^2$ form a linear representation $\bar H$ of $G$. Since $G$ is reductive in characteristic 0, the representation $\bar H$  admits a lifting  $H \subset \cI_{V,p}$, giving a $G$-equivariant collection of maximal-contact elements of order 1.  Writing $U = V(H) \subset W$ we have a codimension-1 equivariant embedding $V\subset U$ with $U$ regular, resulting in a codimension-1 embedding $X \subset Y := [U/G]$.

Any two embeddings are \'etale-locally  isomorphic, with stabilizer-preserving \'etale charts, so  that procedures functorial for  stabilizer-preserving \'etale morphisms glue together. This holds for the relevant  constructions in the present section.

Working with such an embedding, there is an associated divisorial logarithmic structure $\cM\to \cO_Y$ on $Y$ and a closed subgroup $\cG\subset \cI_Y$ of the inertia of $Y$ which is the kernel of the action on the characteristic monoid $\ocM$. The formation of $\cG$ is independent of embeddings and is  functorial for  stabilizer-preserving \'etale morphisms.\footnote{Note that the logarithmic structure itself does not glue even for schemes, see \cite{log-moduli}, Theorem 5.6.} As shown by Rydh \cite[Theorem A.1.3]{ATW-destackification}, there is a canonical coarsening morphism of stacks $Y \to Y\thickslash\, \cG$ rigidifying $\cG$. It is functorial for stabilizer-preserving \'etale morphisms.

We claim that there is a  relative destackification $$\xymatrix{(Y'',X'') \ar[r]\ar[d] & (Y,X) \ar[d]\ar[dr] \\ (Y',X') \ar[r] &  (Y\thickslash\, \cG, X\thickslash\, \cG)\ar[r] & (\overline Y, \oX),}$$  functorial for  stabilizer-preserving \'etale morphisms. In particular $Y'$ is regular, $X'$ is a normal crossings divisor, and the morphism $Y' \to Y \thickslash\ \cG$ is representable.

We note that the work \cite{Bergh},  \cite[Corollary 3.19.1]{Harper}, \cite{Bergh-rydh}, assumes $Y$ is a standard pair, with divisor which is an ordered simple normal-crossings divisor, and $X$ cannot serve for this purpose. Instead we apply relative destackification in Harper's sense to the pair $(Y,\emptyset)$ relative to the coarsening morphism $Y \to Y\thickslash\ \cG$: this is obtained by taking a schematic \'etale presentation $R\double U \to  Y\thickslash\ \cG$, and applying destackification to $U \times_{Y\thickslash\ \cG} Y$ and $R \times_{Y\thickslash\ \cG} Y$. These provide, see \cite[Corollary 3.18.5]{Harper}, descent data for  a destackification diagram
$$\xymatrix{(Y'',E'') \ar[r]\ar[d] & (Y,\emptyset) \ar[d] \\ (Y',E') \ar[r] &  (Y\thickslash\, \cG, \emptyset)}$$
with $Y'$ regular and $Y' \to Y\thickslash\ \cG$ representable. Functoriality guarantees that $E''$ is transversal to $X''$, hence $X'$ remains normal-crossings.
%
%
Functoriality also guarantees that the process is trivial over the locus where $X$ is representable.

\begin{proof}[Proof of Corollary \ref{Cor:poly-circulant}] The group $\cI_{Y \thickslash\ \cG}$ acts faithfully on the characteristic monoid $\ocM_{Y \thickslash\ \cG}$. Since $Y' \to Y \thickslash\ \cG$ is representable, also  $\cI_{Y'}$ acts faithfully on $\ocM_{Y'}$, in particular  $\cI_{X'}$ acts faithfully on $\ocM_{X'}$, as needed. If further  $\cI_X$ is abelian, this implies that $\overline X'$  has
higher pinch point singularities, as needed.
\end{proof}

\appendix

\section{The wilderness}

\subsection{Some ramification theory}
In this appendix we recall some facts about wild quadratic extensions, see \S\ref{quadsec}.

\subsubsection{Unibranch extensions}
Throughout this section $K'/K$ is a finite separable extension of discrete valuation fields of positive residual characteristic $p$ with the associated extension of DVR's $\cO'/\cO$. Also we assume that $K'/K$ is {\em unibranch}, that is, the valuation extends uniquely. By $\nu\:K^\times\to \ZZ= K^\times/\cO^\times$ we denote the valuation of $K$, which extends to $\nu\:K'^\times\to\frac{1}{e}\ZZ$, where $e=e_{K'/K}$ is the ramification index. By $\nu'=e\nu$ we denote the rescaled valuation on $K'$ with group of values $\ZZ$.

Since the extension is separable, $\cO'$ is a finite $\cO$-module. Thus, $\cO'\toisom\cO^n$ for $n=[K':K]$, and hence the fundamental equality $n=ef$ holds, where $f=[k':k]$ is the degree of the residual extension.

\subsubsection{The different}
The module $\Omega_{\cO'/\cO}$ is a finite torsion module, whose length $\delta_{K'/K}$ is the most fundamental invariant of the extension, called the {\em different} of the extension of DVR's or the associated valued fields. We will only need the following particular case:

\begin{lemma}\label{diflem}
Let $K'/K$ be a finite separable unibranched extension of valued fields such that $\cO'=\cO[\alpha]$ and let $f_\alpha$ be the minimal polynomial of $\alpha$ over $K$, then

(i) $\Omega_{\cO'/\cO}=\cO'dt/\cO'f'_\alpha dt$ is a cyclic module.

(ii) $\delta_{K'/K}=\nu'(f'_\alpha(\alpha))$.

(iii) $\Omega_{\cO'/\cO}=\Omega_{\cO'/\cO'_u}$, where $\cO'_u$ is the ring of integers of the maximal unramified subextension.

(iv) $\Omega_{\cO'/\cO}=\Omega_{\hatcO'/\hatcO}$.
\end{lemma}
\begin{proof}
Since $\cO$ is normal, $\cO'=\cO[t]/f_\alpha\cO[t]$, and (i) follows from the first exact sequence for differentials. The other claims follow from (i) easily.
\end{proof}

\begin{rem}
Sometimes it is even more natural to look at the module of logarithmic differentials $\Omega^\rmlog_{\cO'/\cO}$ generated by the logarithmic differentials $\delta x=\frac{dx}{x}$ and its length $\delta^{\log}_{K'/K}$ called the {\em logarithmic different} of the extension. One always has that $\delta^{\log}_{K'/K}=\delta_{K'/K}-e+1$. The log different vanishes if and only if the extension is tame (or log smooth) and $\Omega^\rmlog_{\cO'/\cO}=\Omega^\rmlog_{\cO'/\cO'_t}$, where $\cO'_t$ is the ring of integers of the maximal tame subextension.
\end{rem}

\subsubsection{Monogenicity}
The monogenicity assumption $\cO'=\cO[\alpha]$ is automatically satisfied in two basic cases: (i) if the extension of residue fields $k'/k$ is separable, (ii) when $[K':K]=p$ and the extension is separable (or just $ef=p$). In fact, in the first case one can take any $\alpha=u+\pi'$, where $\pi'$ is a uniformizer of $\cO'$ and $u\in\cO'$ is such that its reduction generates $k'$ over $k$, see \cite[Proposition~9.1]{Frohlich}. 
In the second case one should take $\alpha=\pi'$ if $e_{K'/K}=p$ and $\alpha=u$ if $f_{K'/K}=p$.

\begin{rem}
Informally speaking, the assumption that $k$ is perfect makes the ramification theory 1-dimensional, as indicated by the cyclicity of $\Omega_{\cO'/\cO}$ for any extensions $\cO'/\cO$. This assumption has other interesting consequences too, including the existence of a meaningful Herbrand's function relating upper and lower jumps in ramification groups.
\end{rem}

\subsubsection{Wild extensions of degree $p$}
Assume now that $K'/K$ is wild separable of degree $p$. The cases of $f_{K'/K}=p$ and $e_{K'/K}=p$ should be considered separately, and we will only study the first one. Choose any uniformizer $\pi'\in\cO'$; it annihilates an Eisenstein equation $f_\alpha=t^p+a_1t^{p-1}+\dots+a_p$, where $a_p=\pi$ is a uniformizer of $\cO$ and $a_i\in m_{\cO}$. Then using that $1,\alpha\.\alpha^{p-1}$ form an orthogonal basis of $K'$ over $K$ the value of $\delta=\delta_{K'/K}$ can be computed as $$\delta=\nu'(f'_\alpha(\alpha))=\nu'\left(p\alpha^{p-1}+\sum_{i=1}^{p-1}ia_i\alpha^{p-i-1}\right)=\min_{0\le i\le p}(p\nu(a_i)+p-i-1),$$ where we set $a_0=p$ for shortness. In particular, the minimum is obtained for a single value of $i$ which we denote $r$ and there are two cases:

\begin{itemize}
\item[(i)] $0<r<p$ and then $\delta<p\nu(p)+p-1$ and $(p,\delta+1)=1$.
\item[(ii)] $r=p$ and then $\delta=p\nu(p)+p-1$ and the characteristic is mixed.
\end{itemize}

\begin{rem}
In fact, the divisibility and inequality (in the mixed characteristic case) conditions are the only restrictions on the value of the different. Moreover, to realize any permissible value of the different one can take $K'$ generated by the root $\alpha$ of a trinom $t^p+\pi^nt^r+\pi$ or the binom $t^p+\pi$.

\end{rem}

\subsubsection{A classification in the complete case}
Assume that $k$ is algebraically closed and $\chara(K)=p$. If $K$ is complete, then $K=k((\pi))$ (non-canonically) and $\pi=\sum_{i=p}^\infty c_i\pi'^i$ with $c_i\in k$, $c_p\neq 0$. Deriving this series with respect to $\pi$ one obtains that $\delta$ is the minimal number such that $d=\delta+1$ is prime to $p$ and $c_{d}\neq 0$. It follows that $\sum_{i<d}c_i\pi'^i=\pi_1^p$ and $\sum_{i\ge d}c_i\pi'^i=\pi_2^n$ for uniformizers $\pi_1,\pi_2$ of $K'$. It turns out that playing with $\pi$ one can even achieve that $\pi_1=\pi_2$. Namely, it is proved in \cite{BT}, see a more general \cite[Theorem~4.3.8]{BT} and the explanation in \cite[Remark~4.4.11]{BT}, that there exists uniformizers $\pi$ and $\pi'$ such that $\pi=\pi'^p+\pi'^{\delta+1}$. In fact, this claim is equivalent to the claim that up to automorphisms of both $K$ and $K'$ the isomorphism class of the extension $K'/K$ is determined by the different. Also it is shown that the extension is normal (and hence Artin-Shreier) if and only if $(p-1)|\delta$, and in this case, it is generated by a root of $t^n-t-\pi^{-\delta/(p-1)}$.

\bibliographystyle{amsalpha}
\bibliography{principalization}

\end{document}